\documentclass[a4paper, 12pt]{article}
\usepackage{}

\usepackage{mathrsfs}
\usepackage{epsfig}
\usepackage{amsmath}
\usepackage{amssymb}
\usepackage{latexsym}
\usepackage{amsfonts}
\usepackage{amsthm}

\usepackage[numbers,sort&compress]{natbib}
\usepackage{graphicx}
\ifx\pdfoutput\undefined  \else
 \fi

\usepackage[centerlast]{caption2}

\usepackage{color}

\usepackage{txfonts}
\usepackage{amssymb}
\usepackage{mathrsfs}
\usepackage{amsfonts}

\newtheorem{theorem}{Theorem}[section]
\newtheorem{lemma}[theorem]{Lemma}
\newtheorem{cor}[theorem]{Corollary}
\newtheorem{prop}[theorem]{Proposition}

\newtheorem{remark}[theorem]{Remark}

\usepackage{latexsym,amssymb}
\usepackage{amsmath}

\newcommand{\F}{{\mathbb F}}

\begin{document}

\title{{Erd\H{o}s-Burgess constant of commutative semigroups}}
\author{
Guoqing Wang\\
\small{School of Mathematics Science, Tiangong University, Tianjin, 300387, P. R. China}\\
\small{Email: gqwang1979@aliyun.com}
\\
}
\date{}
\maketitle

\begin{abstract}
Let $\mathcal{S}$ be a nonempty commutative semigroup written additively.
An element $e$ of $\mathcal{S}$ is said to be idempotent if $e+e=e$. The Erd\H{o}s-Burgess constant of the semigroup $\mathcal{S}$ is defined as the smallest positive integer $\ell$ such that any $\mathcal{S}$-valued sequence $T$ of length $\ell$ contain a nonempty subsequence the sum of whose terms is an idempotent of $\mathcal{S}$. We make a study of ${\rm I}(\mathcal{S})$ when
$\mathcal{S}$ is a direct product of arbitrarily many of cyclic semigroups. We give the necessary and sufficient conditions such that ${\rm I}(\mathcal{S})$ is finite, and in particular, we obtain sharp bounds of ${\rm I}(\mathcal{S})$ in case ${\rm I}(\mathcal{S})$ is finite, and determine the precise values of ${\rm I}(\mathcal{S})$ in some cases which unifies some well known results on the precise values of Davenport constant in the setting of commutative semigroups.
\end{abstract}


\section {Introduction}

Let $\mathcal{S}$ be a nonempty semigroup, endowed with a binary associative operation $*$. Let ${\rm E}(\mathcal{S})$ be the set of idempotents of $\mathcal{S}$, where $e\in \mathcal{S}$ is said to be an idempotent if $e*e=e$.
Idempotent is one of central notions in Semigroup Theory and Algebra,
also connects closely with other fields, see \cite{Cohen, Green} for the idempotent theorem in harmonic analysis, see \cite{Lint} for the application in coding theory.
One of our interest to combinatorial properties concerning idempotents in semigroups comes from a question of P. Erd\H{o}s to D.A. Burgess (see \cite{Burgess69} and \cite{Gillam72}), which can be restated as follows.

{\sl Let $\mathcal{S}$ be a finite nonempty semigroup of order $n$. A sequence of terms from $\mathcal{S}$ of length $n$ must contain one or more terms whose product, in some order, is idempotent?}

Burgess \cite{Burgess69} in 1969 gave an answer to this question in the case when $\mathcal{S}$ is commutative or contains only one idempotent. Shortly after, this question was completely affirmed by
D.W.H. Gillam, T.E. Hall and N.H. Williams, who proved the following stronger result.

\noindent \textbf{Theorem A.} (\cite{Gillam72}) \ {\sl Let $\mathcal{S}$ be a finite nonempty semigroup. Any sequence $T$ of terms from $\mathcal{S}$ of length $|T|\geq |\mathcal{S}|-|{\rm E}(\mathcal{S})|+1$ must contain one or more terms such that their product, in the order induced from $T$, is an idempotent of $\mathcal{S}$. In addition, the bound $|\mathcal{S}|-|{\rm E}(\mathcal{S})|+1$ is optimal.}

Very recently, the author of this manuscript extends the result to infinite semigroups.

\noindent \textbf{Theorem B.} (\cite{wangidempotent}, Theorem 1.1) \ {\sl Let $\mathcal{S}$ be a nonempty semigroup such that $|\mathcal{S}\setminus {\rm E}(\mathcal{S})|$ is finite. Any sequence $T$ of terms from $\mathcal{S}$ of length $|T|\geq|\mathcal{S}\setminus {\rm E}(\mathcal{S})|+1$ must contain one or more terms such that their product, in the order induced from $T$, is an idempotent of $\mathcal{S}$. }

Hence, one combinatorial constant concerning idempotents in commutative semigroups was aroused which can be stated as Definition C below.

\noindent $\bullet$ In what follows, since we deal with only commutative semigroups, we always let $\mathcal{S}$ be a nonempty commutative semigroup written additively, for which the operation is denoted by +.

\noindent \textbf{Definition C.} (\cite{wangidempotent}, Definition 4.1) \ {\sl Let $T$ be a sequence of terms from $\mathcal{S}$. We call $T$ {\bf idempotent-sum free} if $T$ contains no one or more terms with their sum being an idempotent of $\mathcal{S}$. Define $\textsc{I}(\mathcal{S})$, which is called the {\bf Erd\H{o}s-Burgess constant} of the semigroup $\mathcal{S}$, as
\begin{equation}\label{equation in definition of I(S)}
\textsc{I}(\mathcal{S})={\rm sup}\ \{|T|+1: T \mbox{ takes over all idempotent-sum free sequences of terms from } \mathcal{S}\}.
\end{equation}}

When the commutative semigroup $\mathcal{S}$ happens to be a finite abelian group, the Erd\H{o}s-Burgess constant reduces to one classical combinatorial constant, the Davenport constant. The Davenport constant of a finite abelian group $G$, denoted ${\rm D}(G)$, is defined as the smallest positive integer $\ell$  such that every sequence of terms from $G$ of length at least $\ell$ contains one or more terms with the product being the identity element of $G$. This invariant was popularized by H. Davenport in the 1960's, notably for its link with algebraic number theory (as reported in \cite{Olson1}). However, it seems that it was K. Rogers \cite{Rogers} to write the first paper to deal with the Davenport constant. This invariant ${\rm D}(G)$ has been investigated extensively in the past over 50 years, and  found applications in other areas, including Factorization Theory of Algebra (see \cite{CziszterDoGerolding, GRuzsa, GH}), Classical Number Theory, Graph Theory, and Coding Theory.  For example, the Davenport constant has been applied by W.R. Alford, A. Granville and C. Pomerance \cite{AGP} to prove that there are infinitely many Carmichael numbers, by N. Alon \cite{AlonJctb} to prove the existence of regular subgraphs, and by L.E. Marchan, O. Ordaz, I. Santos and W.A. Schmid \cite{MaOrSaSc} to establish a link between two variant Davenport constants and problems of linear codes.
What is more important, a lot of researches were motivated by the Davenport constant together with the celebrated EGZ Theorem obtained by P. Erd\H{o}s, A. Ginzburg and A. Ziv \cite{EGZ} in 1961 on additive properties of sequences in groups, which have been developed into a branch, called Zero-sum Theory (see \cite{GaoGeroldingersurvey} for a survey).

As a consequence of the Fundamental Theorem for finite abelian groups, any nontrivial finite abelian group can be written as the direct sum  $\mathbb{Z}_{n_1}\oplus\cdots\oplus \mathbb{Z}_{n_r}$ of cyclic groups  $\mathbb{Z}_{n_1},\ldots, \mathbb{Z}_{n_r}$ with $1 < n_1 \mid \cdots \mid n_r$.
Let us denote by
$${\rm d}^*(G) = \sum\limits_{i=1}^r (n_i-1).$$ In the 1960s, D. Kruyswijk \cite{EmdeBoas007} and J.E. Olson \cite{Olson2} independently gave the inequality
${\rm D}(G)\geq 1+{\rm d}^*(G)$.  On the other hand, P. Van Emde Boas and D. Kruyswijk \cite{EmdeBoas} and R. Meshulam \cite{Meshu} proved
that
$${\rm D}(G)\leq n_r+ n_r {\rm log}(\frac{|G|}{n_r}).$$ A lot of efforts were made to find the precise values of Davenport constant for finite abelian groups. However, up to date, besides for the groups of types given in Theorem D, the precise value of this constant was known only for groups of specific forms such as $\mathbb{Z}_{2}\oplus \mathbb{Z}_2 \oplus \mathbb{Z}_{2d}$ (see \cite{EmdeBoas007}), or $\mathbb{Z}_3\oplus \mathbb{Z}_3 \oplus \mathbb{Z}_{3d}$ (see \cite{Bhowmik}), etc. Even to determine the precise value of ${\rm D}(G)$ in the case when $G$ is a direct sum of three finite cyclic groups remains open for over 50 years (see \cite{GaoGeroldingersurvey}, Conjecture 3.5).

\noindent \textbf{Theorem D.} (see \cite{Olson1,Olson2,EmdeBoas007}) \ {\sl Let $G=\mathbb{Z}_{n_1}\oplus\cdots\oplus \mathbb{Z}_{n_r}$ with $1 < n_1 \mid \cdots \mid n_r$. Then ${\rm D}(G)= 1+{\rm d}^*(G)$ in case that one of the following conditions holds.

$\bullet$ $r\leq 2$;

$\bullet$ $G$ is a $p$-group;
}

For the  progress about ${\rm D}(G)$ the reader may
consult \cite{GLP12, GeroldingerScheider, Girard, PlagneSchmid, Schmid} and (\cite{Grynkiewiczmono}, P. 270). Recently, some results were obtained \cite{wangDavenportII,  wangAddtiveirreducible, wang-zhang-qu, wang-zhang-wang-qu} on additive properties of sequences in semigroups associated with Davenport constant.
Because the identity element is the unique idempotent in a group, by definition we have
\begin{equation}\label{equation I is Davenport for groups}
\textsc{I}(G)={\rm D}(G) \ \ \mbox{ for any finite abelian group } G.
 \end{equation}
In this manuscript, by studying the Erd\H{o}s-Burgess constant we hope to
unify the researches on additive properties concerning idempotents and Davenport constant and to extend zero-sum problems into the realm of commutative semigroups.
In Section 2, we introduce necessary notions. In Section 3, we
make a study of
the Erd\H{o}s-Burgess constant for the direct product of cyclic semigroups $\prod\limits_{i\in R} \langle g_i\rangle$ where $R$ is an arbitrary set and
$\langle g_i\rangle$ is a cyclic semigroup generated by $g_i$ for $i\in R$.
We give necessary and sufficient conditions such that the Erd\H{o}s-Burgess constant ${\rm I}(\prod\limits_{i\in R} \langle g_i\rangle)$ exists. In the case when ${\rm I}(\prod\limits_{i\in R} \langle g_i\rangle)$ exists, we give sharp lower and upper bounds of that, and determine the precise value of ${\rm I}(\prod\limits_{i\in R} \langle g_i\rangle)$ with some constraints which unifies the results on Davenport constant shown in Theorem D into the product of cyclic semigroups.
Moreover, the extremal structure of that semigroup in which the given bound is attained was also discussed preliminarily.

\section{Notions}

For integers $a,b\in \mathbb{Z}$, we set $[a,b]=\{x\in \mathbb{Z}: a\leq x\leq b\}$. For a real number $x$, we denote by $\lfloor x\rfloor$ the largest integer that is less
than or equal to $x$, and by $\lceil x\rceil$ the smallest integer that is greater than or equal to $x$.

The identity element of the commutative semigroup $\mathcal{S}$, denoted $0_{\mathcal{S}}$ (if exists), is the unique element $z$ of
$\mathcal{S}$ such that $z+a=a$ for every $a\in \mathcal{S}$.  For any positive integer $m$ and any element $a\in \mathcal{S}$, we denote by $ma$ the sum $\underbrace{a+\cdots+a}\limits_{m}$. We use idempotent to mean an element $e$ such that $e+e=e$ with the addition of $\mathcal{S}$, and let ${\rm E}(\mathcal{S})=\{e\in\mathcal{S}: e \mbox{ is an idempotent} \}$.
Let $X$ be a subset of $\mathcal{S}$. We say that $X$ generates $\mathcal{S}$, or the elements of $X$ are generators of $\mathcal{S}$, provided that every element $a\in \mathcal{S}$ can be written as a finite sum of one or more elements (repetition is allowed) from $X$, in which case we write $S=\langle X\rangle$. If S happens to be a monoid (a semigroup with an identity), we will understand the identity element of this monoid as a trivial finite sum of elements from $X$ (of length $0$).
In particular, we use $\langle x\rangle$ in place of $\langle X\rangle$ if $X=\{x\}$ is a singleton, and we say that $\mathcal{S}$ is a cyclic semigroup if it is generated by a single element. For any element $x\in \mathcal{S}$ such that $\langle x\rangle$ is finite, the least integer $k>0$ such that $kx=tx$ for some positive integer $t\neq k$ is the {\sl index} of $x$,  then the least integer $n>0$ such that $(k+n)x=k x$ is the {\sl period} of $x$. We denote a finite cyclic semigroup of index $k$ and period $n$ by $C_{k; n}$.

\noindent $\bullet$ {\sl Note that if $k=1$ the semigroup $C_{k; n}$ reduces to be a cyclic group of order $n$ which is isomorphic to the additive groups $\mathbb{Z}_{n}$ of integers modulo $n$.}

We also need to introduce notation and terminologies on sequences over semigroups and follow the notation of A. Geroldinger, D.J. Grynkiewicz and
others used for sequences over groups (cf. [\cite{Grynkiewiczmono}, Chapter 10] or [\cite{GH}, Chapter 5]). Let ${\cal F}(\mathcal{S})$
be the free commutative monoid, multiplicatively written, with basis
$\mathcal{S}$. We denote multiplication in $\mathcal{F}(\mathcal{S})$ by the boldsymbol $\cdot$ and we use brackets for all exponentiation in $\mathcal{F}(\mathcal{S})$. By $T\in {\cal F}(\mathcal{S})$, we mean $T$ is a sequence of terms from $\mathcal{S}$ which is
unordered, repetition of terms allowed. Say
$T=a_1a_2\cdot\ldots\cdot a_{\ell}$ where $a_i\in \mathcal{S}$ for $i\in [1,\ell]$.
The sequence $T$ can be also denoted as $T=\mathop{\bullet}\limits_{a\in \mathcal{S}}a^{[{\rm v}_a(T)]},$ where ${\rm v}_a(T)$ is a nonnegative integer and
means that the element $a$ occurs ${\rm v}_a(T)$ times in the sequence $T$. By $|T|$ we denote the length of the sequence, i.e., $|T|=\sum\limits_{a\in \mathcal{S}}{\rm v}_a(T)=\ell.$ By $\varepsilon$ we denote the
{\sl empty sequence} in $\mathcal{S}$ with $|\varepsilon|=0$. We call $T'$
a subsequence of $T$ if ${\rm v}_a(T')\leq {\rm v}_a(T)\ \ \mbox{for each element}\ \ a\in \mathcal{S},$ denoted by $T'\mid T,$ moreover, we write $T^{''}=T\cdot  T'^{[-1]}$ to mean the unique subsequence of $T$ with $T'\cdot T^{''}=T$.  We call $T'$ a {\sl proper} subsequence of $T$ provided that $T'\mid T$ and $T'\neq T$. In particular, the  empty sequence  $\varepsilon$ is a proper subsequence of every nonempty sequence. We say $T_1,\ldots,T_{m}$ are {\bf disjoint subsequences} of $T$ provided that $T_1\cdot \ldots\cdot T_{m}\mid T$.
Let $\sigma(T)=a_1+\cdots +a_{\ell}$ be the sum of all terms from $T$.
We call $T$ a {\bf zero-sum} sequence provided that $\mathcal{S}$ is a monoid and $\sigma(T)=0_{\mathcal{S}}$.
In particular,
if $\mathcal{S}$ is a monoid,  we allow $T=\varepsilon$ to be empty and adopt the convention
that $\sigma(\varepsilon)=0_\mathcal{S}.$
We say  the sequence $T$ is
\begin{itemize}
     \item a {\bf zero-sum free sequence} if $T$ contains no nonempty zero-sum subsequence;
     \item a {\bf minimal zero-sum sequence} if $T$ is a nonempty zero-sum sequence and and $T$ contains no nonempty proper zero-sum subsequence;
     \item an {\bf idempotent-sum sequence} if $\sigma(T)\in {\rm E}(\mathcal{S})$;
     \item an {\bf idempotent-sum free sequence} if $T$ contains no nonempty idempotent-sum subsequence;
\end{itemize}

It is worth remarking that when the commutative semigroup $\mathcal{S}$ is an abelian group, the notion {\sl zero-sum sequence} [resp. zero-sum free sequence] and {\sl idempotent-sum sequence} [resp. idempotent-sum free sequence] make no difference.

Let $\{A_i: i\in R\}$ be a family of sets indexed by a (nonempty) set $R$. The Cartesian product of $A_i$ is the set of all functions $f:R\rightarrow \bigcup\limits_{i\in R} A_i$ such that $f(i)\in A_i$ for all $i\in R$. It is denoted $\prod\limits_{i\in R} A_i.$ Let $\{\mathcal{S}_i: i\in R\}$ be a family (possibly infinite) of {\sl commutative} semigroups with operation $+_{\mathcal{S}_i}$.
Define a binary {\sl commutative} operation on the Cartesian product (of sets) $\prod\limits_{i\in R}\mathcal{S}_i$ as follows. If $f,g\in \prod\limits_{i\in R}\mathcal{S}_i$ (that is, $f,g: R\rightarrow \bigcup\limits_{i\in R}\mathcal{S}_i$, and $f(i),g(i)\in \mathcal{S}_i$ for each $i$), then $f+g:R\rightarrow \bigcup\limits_{i\in R}\mathcal{S}_i$ is the function given by $f+g:i\mapsto f(i)+_{\mathcal{S}_i}g(i)$, in particular, in the case when $R$ is finite, we could identify $f\in \prod\limits_{i\in R}\mathcal{S}_i$ with its image $\{x_i\}_{i\in R}$ ($x_i=f(i)$ for each $i\in R$), and the binary operation `+' in $\prod\limits_{i\in R}\mathcal{S}_i$ is the familiar component-wise addition: $\{x_i\}_{i\in R}+\{y_i\}_{i\in R}=\{x_i+_{\mathcal{S}_i} y_i\}_{i\in R}$. Then  $\prod\limits_{i\in R}\mathcal{S}_i$ together with this binary operation, is called the direct product of the family of semigroups $\{\mathcal{S}_i: i\in R\}$. If $R=[1,n]$, then $\prod\limits_{i\in R}\mathcal{S}_i$ is usually denoted $\mathcal{S}_1\times \mathcal{S}_2\times \cdots\times \mathcal{S}_n$.

For any nonempty set $R$, let $\mathcal{S}=\prod\limits_{i\in R}\langle g_i\rangle$ be the direct product of a family of cyclic semigroups $\{\langle g_i\rangle: i\in R\}$.
By Proposition 4.2 in \cite{wangidempotent}, if there exists some $i\in R$ such that $\langle g_i\rangle$ is infinite, then ${\rm I}(\mathcal{S})$ is infinite. So we need only to consider the case when all $\langle g_i\rangle$ are finite. Hence, we shall denote
$\mathcal{S}=\prod\limits_{i\in R}{\rm C}_{k_i;  n_i}$  where the finite cyclic semigroup ${\rm C}_{k_i;  n_i}$ is generated by $g_i$ with index $k_i>0$ and period $n_i>0$ for each $i\in R$.
For any element $a$ of $\mathcal{S}$ and any $i\in R$, let ${\rm ind}_{g_i}(a(i))$ be the least positive integer $t_i$ such that $t_i g_i=a(i)$.
Let $$G_\mathcal{S}=\prod\limits_{i\in R}\mathbb{Z}_{n_i}$$ be the direct product of a family of additive groups of integers modulo $n_i$, which is the largest group contained in $\mathcal{S}$. Define a map $\psi:\mathcal{S}\rightarrow G_\mathcal{S}$ as follows.
For any $a\in \mathcal{S}$, let $\psi(a)\in G_\mathcal{S}$ be such that $\psi(a)(i)={\rm ind}_{g_i} (a(i)) \ {\rm mod}\  n_i$ for each $i\in R$. We could extend $\psi$ to the map $\Psi:\mathcal{F}(\mathcal{S})\rightarrow \mathcal{F}(G_{\mathcal{S}})$ given by $\Psi: T\mapsto \mathop{\bullet}\limits_{a\mid T} \psi(a)$ for any sequence $T\in \mathcal{F}(\mathcal{S})$.

\section{Results}

We begin this section with some necessary lemmas.

\begin{lemma}\label{lemma I(sub)< I(S)} Let $S$ be a commutative semigroup and $S'$ a subsemigroup of $S$. If ${\rm I}(\mathcal{S})$ is finite, then ${\rm I}(\mathcal{S}')$ is finite and ${\rm I}(\mathcal{S}')\leq {\rm I}(\mathcal{S})$.
\end{lemma}

\begin{proof} The conclusion follows immediately from the fact that any idempotent-sum free sequence of terms from $S'$ is also an idempotent-sum free sequence of terms from $S$.
\end{proof}

\begin{lemma} (Folklore) \label{lemma n-1 or n zero-sum free}\ Let $n\geq 2$,  and let $T\in \mathcal{F}(\mathbb{Z}_n)$ be a sequence. Suppose that $T$ is either a zero-sum free sequence of length $n-1$ or a minimal zero-sum sequence of length $n$. Then the sequence $T$ admits only one distinct value, which is a generator of the group $\mathbb{Z}_n$.
 \end{lemma}

\begin{lemma}\label{Lemma decomposition} \ Let $n\geq 2$, and let $T\in \mathcal{F}(\mathbb{Z}_n\setminus \{0\})$ of length at least $n-1$. Let $U$ be one of the longest zero-sum subsequences of $T$. If $|T\cdot U^{[-1]}|=n-1$ then $T$ admits only one distinct value, which is a generator of the group $\mathbb{Z}_n$.
\end{lemma}

\begin{proof}  We take disjoint minimal zero-sum subsequences of $T$, say $U_1,\ldots, U_{\ell}$, with $\ell$ being {\bf maximal}. Let $U_0=T\cdot (U_1\cdot \ldots\cdot U_{\ell})^{[-1]}$. By the maximality of $\ell$, we have $U_0$ is zero-sum free. It follows that $|U_0|\leq {\rm D}(\mathbb{Z}_n)-1=n-1$, which implies that $U=U_1\cdot \ldots\cdot U_{\ell}$ is one of the longest zero-sum subsequence of $T$ and $|U_0|=n-1$. It follows from Lemma \ref{lemma n-1 or n zero-sum free} that $$U_0=z^{[n-1]}$$ where $z$ is some generator of $\mathbb{Z}_n$. Now it suffices to assume that $|T|\geq n$, i.e., $\ell>0$, and prove that
$\alpha=z$ for any term $\alpha \mid U$.

Take arbitrary $\theta\in [1,\ell]$ and take an arbitrary term $x$ of $U_{\theta}$. Since $|x\cdot U_0|=n={\rm D}(\mathbb{Z}_n)$ and $U_0$ is zero-sum free, it follows that
$x\cdot U_0$ contains a minimal zero-sum subsequence  $U_{\theta}^{'}$  with $x\mid U_{\theta}^{'}.$
Let $U_0^{'}=(U_{\theta}\cdot U_0) \cdot {U_{\theta}^{'}}^{[-1]}$. Observe that $U_1,\ldots, U_{\theta-1}, U_{\theta}^{'},U_{\theta+1},\ldots,U_{\ell}$ are disjoint minimal zero-sum subsequences of $T$, and that $U_0^{'}=T\cdot (U_1\cdot \ldots\cdot  U_{\theta-1}\cdot U_{\theta}^{'}\cdot U_{\theta+1}\cdot \ldots\cdot U_{\ell})^{[-1]}$.
By the maximality of $\ell$, we see that $U_0^{'}$ is also a zero-sum free sequence and
\begin{equation}\label{equation length of U0'}
|U_0^{'}|=n-1.
\end{equation}

Suppose $|U_{\theta}^{'}|=n$. Then $U_{\theta}^{'}=x\cdot U_0$. It follows from Lemma \ref{lemma n-1 or n zero-sum free} that $x=z.$

Suppose $|U_{\theta}^{'}|<n.$ Observe that $U_0^{'}=((x\cdot U_0)\cdot {U_{\theta}^{'}}^{[-1]})\cdot (U_{\theta}\cdot x^{[-1]})$  and
$(x\cdot U_0) \cdot {U_{\theta}^{'}}^{[-1]}$ is a nonempty subsequence of $U_0$. Since all terms of $T$ are nonzero, it follows that
\begin{equation}\label{equation |Utheta|>1}
|U_{\theta}|>1
\end{equation}
and $|U_{\theta}\cdot x^{[-1]}|>0$.
By \eqref{equation length of U0'} and Lemma \ref{lemma n-1 or n zero-sum free}, we derive that $y=z$ for every term  $y\mid U_{\theta}\cdot x^{[-1]}.$
By \eqref{equation |Utheta|>1}, we have that as $x$ takes every term of $U_{\theta}$, so does $y$.  By the arbitrariness of choosing $\theta$ from $[1,\ell]$ and the arbitrariness of choosing $x$ from $U_{\theta}$, we have the lemma proved.
\end{proof}

\begin{lemma}\label{Lemma cyclic semigroup} (\cite{Grillet monograph},  Chapter I, Lemma 5.7, Proposition 5.8, Corollary 5.9) \ Let $\mathcal{S}=C_{k; n}$ be a finite cyclic semigroup generated by the element $x$. Then  $\mathcal{S}=\{x,\ldots,k x,(k+1)x,\ldots,(k+n-1)x\}$
with
$$\begin{array}{llll} & ix+jx=\left \{\begin{array}{llll}
               (i+j)x, & \mbox{ if } \  i+j \leq  k+n-1;\\
                tx, &  \mbox{ if }  \ i+j \geq k+n, \ \mbox{ where} \\
                &  \ \ \ \ \ \ \ \ \ \  k\leq t\leq k+n-1 \ \mbox{ and } \ t\equiv i+j\pmod{n}. \\
              \end{array}
           \right. \\
\end{array}$$
Moreover,

\noindent (i) \  there exists a unique idempotent, $\ell x$, in the cyclic semigroup $\langle x\rangle$, where $$\ell\in [k,k+n-1] \  \mbox{ and }\  \ell\equiv 0\pmod {n};$$

\noindent (ii) \  $\{k x,(k+1)x,\ldots,(k+n-1)x\}$ is a cyclic subgroup of $\mathcal{S}$ isomorphic to the additive group $\mathbb{Z}_{n}$ of integers modulo $n$.
\end{lemma}

By Lemma \ref{Lemma cyclic semigroup}, it is easy to derive the following.

\begin{lemma}\label{Lemma product condition containing idmepotent} \  For any nonempty set $R$, let $\mathcal{S}=\prod\limits_{i\in R}{\rm C}_{k_i;  n_i}$ where ${\rm C}_{k_i;  n_i}=\langle g_i\rangle$ for each $i\in R$. Then there exists a unique idempotent $e$ in $\mathcal{S}$, where $${\rm ind}_{g_i}(e(i))\in [k_i,k_i+n_i-1] \  \mbox{ and }\  {\rm ind}_{g_i}(e(i))\equiv 0\pmod {n_i}$$ for each $i\in R$. In particular, for any sequence  $W\in \mathcal{F}(\mathcal{S})$, $\sigma(W)$ is equal to the unique idempotent $e$ in $\mathcal{S}$ if, and only if,
$\sum\limits_{a\mid W}{\rm ind}_{g_i}(a(i))\geq \left\lceil\frac{k_i}{n_i}\right\rceil n_i$ and $\sum\limits_{a\mid W}{\rm ind}_{g_i}(a(i))\equiv 0\pmod{n_i}$ for all $i\in R$.
\end{lemma}

Note that in Lemma \ref{Lemma product condition containing idmepotent}, the condition that $\sum\limits_{a\mid W}{\rm ind}_{g_i}(a(i))\equiv 0\pmod{n_i}$ for all $i\in R$ is equivalent to that $\Psi(W)$ is a zero-sum sequence in the group $G_{\mathcal{S}}$.

\begin{theorem}\label{theorem in product of cyclic semigroups} \  For any nonempty set $R$, let $\mathcal{S}=\prod\limits_{i\in R}{\rm C}_{k_i;  n_i}$ where $k_i,n_i\geq 1$ for any $i\in R$. Let $R_1=\{i\in R: n_i>1\}$.  Then ${\rm I}(\mathcal{S})$ is finite if and only if both $|R_1|$ and ${\rm sup}\ \{k_i:i\in R\setminus R_1\}$ are finite. Moreover, if ${\rm I}(\mathcal{S})$ is finite, then

\noindent (i)\  $${\rm I}(\mathcal{S})\geq \max\left({\rm sup}\left\{(\left\lceil\frac{k_i}{n_i}\right\rceil-1)n_i: i\in R\right\}+1+ \sum\limits_{i\in  R_1}(n_i-1), \ \ {\rm sup}\left\{\left\lceil\frac{k_i}{n_i}\right\rceil-1: i\in R\right\}+{\rm D(G_{\mathcal{S}})} \right);$$

\noindent (ii)\  $$ {\rm I}(\mathcal{S})\leq  {\rm sup}\left\{(\left\lceil\frac{k_i}{n_i}\right\rceil-1)n_i: i\in R\right\}+{\rm D(G_{\mathcal{S}})},$$ and if ${\rm D(G_{\mathcal{S}})}=1+\sum\limits_{i\in R_1}(n_i-1)$ then the equality ${\rm I}(\mathcal{S})={\rm sup}\left\{(\left\lceil\frac{k_i}{n_i}\right\rceil-1)n_i: i\in R\right\}+{\rm D(G_{\mathcal{S}})}$ holds;

\noindent (iii) \ Suppose that $\gcd(n_i,\  n_j)=1$ for any distinct $i,j$ of $R$.
Then $${\rm I}(\mathcal{S})={\rm sup}\left\{(\left\lceil\frac{k_i}{n_i}\right\rceil-1)n_i: i\in R\right\}+{\rm D(G_{\mathcal{S}})}$$ holds if, and only if, there exists some $\epsilon\in R$ such that $(\left\lceil\frac{k_{\epsilon}}{n_{\epsilon}}\right\rceil-1)n_{\epsilon}={\rm sup}\left\{(\left\lceil\frac{k_i}{n_i}\right\rceil-1)n_i: i\in R\right\}$, with $\epsilon\notin R_1$ or $\prod\limits_{i\in R_1\setminus \{\epsilon\}}n_i$ divides $\left\lceil\frac{k_{\epsilon}}{n_{\epsilon}}\right\rceil-1$.
\end{theorem}

\begin{proof} \ Say ${\rm C}_{k_i;  n_i}=\langle g_i\rangle$ for each $i\in R$.

(i). For any $i\in R$, let $e_i\in \mathcal{S}$ be such that ${\rm ind}_{g_i}(e_i(i))=1$ and ${\rm ind}_{g_j}(e_i(j))=n_j$ when $j\in R\setminus \{i\}$. Take arbitrary  $r\in R$, and take an arbitrary finite subset $X\subseteq R_1$.
Let $$T_1=e_r^{\left[(\left\lceil\frac{k_r}{n_r}\right\rceil-1)n_r\right]}\cdot(\mathop{\bullet}\limits_{i\in X}e_i^{\left[n_i-1\right]}).$$

\noindent \textbf{Claim A.} \ The sequence $T_1$ is idempotent-sum free.

\noindent {\sl Proof of Claim A.} \ Assume to the contrary that $T_1$ contains a nonempty idempotent-sum subsequence
$L$.

Suppose that there exists some $\delta\in X\setminus \{r\}$ such that ${\rm v}_{e_{\delta}}(L)>0$, i.e., ${\rm v}_{e_{\delta}}(L)\in [1, n_{\delta}-1]$, then $\sum\limits_{a\mid L}{\rm ind}_{g_{\delta}}(a(\delta))=
{\rm v}_{e_{\delta}}(L) \ {\rm ind}_{g_{\delta}}(e_{\delta}(\delta))+
\sum\limits_{a\mid L\cdot e_{\delta}^{[-{\rm v}_{e_{\delta}}(L)]}}{\rm ind}_{g_{\delta}}(a(\delta))\equiv {\rm v}_{e_{\delta}}(L) \ {\rm ind}_{g_{\delta}}(e_{\delta}(\delta))\not \equiv 0\pmod {n_{\delta}}.$ By Lemma \ref{Lemma product condition containing idmepotent}, $L$ is not an idempotent-sum sequence, which is absurd.  Hence,  $L=e_r^{m}$ with $1\leq m\leq (\left\lceil\frac{k_r}{n_r}\right\rceil-1)n_r+(n_r-1)=\left\lceil\frac{k_r}{n_r}\right\rceil n_r-1$.

Since $\sum\limits_{a\mid L}{\rm ind}_{g_r}(a(r))=m \ {\rm ind}_{g_r}(e_r(r))=m<\left\lceil\frac{k_r}{n_r}\right\rceil n_r$, it follows from Lemma \ref{Lemma product condition containing idmepotent}  that $L$ is not an idempotent-sum sequence, a contradiction. This proves Claim A.
\qed

Take an arbitrary zero-sum free sequence $\widetilde{U}\in \mathcal{F}(G_{\mathcal{S}})$. Let $U\in\mathcal{F}(\mathcal{S})$ be such that $\Psi(U)=\widetilde{U}$.
Let $\mu\in \mathcal{S}$ be such that ${\rm ind}_{g_i}(\mu(i))=n_i$ for all $i\in R$.
Take arbitrary $t\in R$. Let
$$T_2=\mu^{\left[\left\lceil\frac{k_t}{n_t}\right\rceil-1\right]}\cdot U.$$

\noindent \textbf{Claim B.} \ The sequence $T_2$ is idempotent-sum free.

\noindent {\sl Proof of Claim B.} \ Assume to the contrary that $T_2$ contains a nonempty idempotent-sum subsequence
$L$, say $L=\mu^{\beta}\cdot W$ with $\beta\in [0, \left\lceil\frac{k_t}{n_t}\right\rceil-1]$ and $W\mid U$.

If $W=\varepsilon$ is an empty sequence, then $\sum\limits_{a\mid L}{\rm ind}_{g_t}(a(t))=\beta \ n_t<\left\lceil\frac{k_t}{n_t}\right\rceil n_t$, a contradiction with $L$ being idempotent-sum by Lemma \ref{Lemma product condition containing idmepotent}. Hence, $W$ is a nonempty subsequence of $U$. By the choice of $\widetilde{U}$,  we have that $\sum\limits_{a\mid W}{\rm ind}_{g_{\eta}}(a(\eta))\not\equiv 0\pmod {n_{\eta}}$ for some $\eta\in R$.
It follows that
$\sum\limits_{a\mid L}{\rm ind}_{g_{\eta}}(a(\eta))=\beta \ n_{\eta}+\sum\limits_{a\mid W}{\rm ind}_{g_{\eta}}(a(\eta))\not\equiv 0\pmod {n_{\eta}}$, a contradiction with $L$ being idempotent-sum by Lemma \ref{Lemma product condition containing idmepotent}. This proves Claim B.
\qed

By Claim A,  and by the arbitrariness of choosing $r\in R$ and $X\subseteq  R_1$, we conclude that
\begin{equation}\label{equation I(S)>=in general}
{\rm I}(\mathcal{S})\geq {\rm sup}\left\{(\left\lceil\frac{k_r}{n_r}\right\rceil-1)n_r: r\in R\right\}+1+{\rm sup}\ \{\sum\limits_{i\in X}(n_i-1): X\mbox{ takes all finite subsets of } R_1\}.
\end{equation}
By Claim B, and by the arbitrariness of choosing $t\in R$ and the sequence $\widetilde{U}$,  we conclude that
\begin{equation}\label{equation I(S)<=in general}
{\rm I}(\mathcal{S})\geq {\rm sup}\left\{\left\lceil\frac{k_t}{n_t}\right\rceil-1: t\in R\right\}+1+{\rm sup}\ \{|\widetilde{U}|: \widetilde{U} \mbox{ takes all zero-sum free sequences of } \mathcal{F}(G_{\mathcal{S}}) \}.
\end{equation}
By \eqref{equation I(S)>=in general}, we derive that if ${\rm I}(\mathcal{S})$ is finite then $|R_1|$ is finite,  and then combined with \eqref{equation I(S)<=in general}, we derive that if ${\rm I}(\mathcal{S})$ is finite then ${\rm sup}\left\{\left\lceil\frac{k_t}{n_t}\right\rceil-1: t\in R\right\}$ is finite, and so ${\rm sup}\left\{k_t-1: t\in R\setminus R_1\right\}={\rm sup}\left\{\left\lceil\frac{k_t}{n_t}\right\rceil-1: t\in R\setminus R_1\right\}$ is finite, equivalently, ${\rm sup}\left\{k_t: t\in R\setminus R_1\right\}$ is finite. Moreover,
combined with \eqref{equation in definition of I(S)} and \eqref{equation I is Davenport for groups}, we derive that the inequality in Conclusion (i) holds in the case when ${\rm I}(\mathcal{S})$ is finite.

(ii). Suppose that both $|R_1|$ and ${\rm sup}\ \{k_i:i\in R\setminus R_1\}$ are finite. Then ${\rm sup}\left\{(\left\lceil\frac{k_i}{n_i}\right\rceil-1)n_i: i\in R\right\}$ is finite, and $G_{\mathcal{S}}=\prod\limits_{i\in R}\mathbb{Z}_{n_i}\cong \prod\limits_{i\in R_1}\mathbb{Z}_{n_i}$ is a finite abelian group. Let $T\in \mathcal{F}(\mathcal{S})$ be a sequence of length $|T|={\rm sup}\left\{(\left\lceil\frac{k_i}{n_i}\right\rceil-1)n_i: i\in R\right\}+{\rm D}(G_{\mathcal{S}})$. Take a nonempty subsequence $L$ of $T$ such that $\Psi(L)$ is a zero-sum sequence over $G_{\mathcal{S}}$, i.e.,
\begin{equation}\label{equation L is zero-sum in G}
\sum\limits_{a\mid L}{\rm ind}_{g_i}(a(i))\equiv 0\pmod {n_i}  \mbox{ for all } \ i\in R,
\end{equation}
with $|L|$ being maximal.
By the maximality of $|L|$, we have that $|T\cdot L^{[-1]}|\leq {\rm D(G_{\mathcal{S}})}-1$, equivalently, $|L|\geq {\rm sup}\left\{(\left\lceil\frac{k_i}{n_i}\right\rceil-1)n_i: i\in R\right\}+1$. It follows that for all $i\in R$, $\sum\limits_{a\mid L}{\rm ind}_{g_i}(a(i))\geq |L|\geq  (\left\lceil\frac{k_i}{n_i}\right\rceil-1)n_i+1$ and from \eqref{equation L is zero-sum in G} that $\sum\limits_{a\mid L}{\rm ind}_{g_i}(a(i))\geq  \left\lceil\frac{k_i}{n_i}\right\rceil n_i$.
By Lemma \ref{Lemma product condition containing idmepotent}, we have that the sequence $T$ is not idempotent-sum free. Therefore, we conclude that ${\rm I}(\mathcal{S})$ is finite and
${\rm I}(\mathcal{S})\leq  {\rm sup}\left\{(\left\lceil\frac{k_i}{n_i}\right\rceil-1)n_i: i\in R\right\}+{\rm D}(G_{\mathcal{S}})$, in the case when both $|R_1|$ and ${\rm sup}\ \{k_i:i\in R\setminus R_1\}$ are finite.  Futhermore, combined with Conclusion (i), we have that if ${\rm D(G_{\mathcal{S}})}=1+\sum\limits_{i\in R_1}(n_i-1)$ then ${\rm I}(\mathcal{S})={\rm sup}\left\{(\left\lceil\frac{k_i}{n_i}\right\rceil-1)n_i: i\in R\right\}+{\rm D}(G_{\mathcal{S}})$. This proves Conclusion (ii).

(iii). Let $N=\prod\limits_{i\in R_1} n_i$.
Note that $G_{\mathcal{S}}\cong \mathbb{Z}_{N}$. By Theorem D, ${\rm D(G_{\mathcal{S}})}=N$.

We first consider the sufficiency.

Suppose that there exists some $\epsilon\in R\setminus R_1$ such that $(\left\lceil\frac{k_{\epsilon}}{n_{\epsilon}}\right\rceil-1)n_{\epsilon}={\rm sup}\left\{(\left\lceil\frac{k_i}{n_i}\right\rceil-1)n_i: i\in R\right\}$. Since $n_{\epsilon}=1$,
we see ${\rm sup}\left\{\left\lceil\frac{k_i}{n_i}\right\rceil-1: i\in R\right\}\geq \left\lceil\frac{k_{\epsilon}}{n_{\epsilon}}\right\rceil-1=(\left\lceil\frac{k_{\epsilon}}{n_{\epsilon}}\right\rceil-1)n_{\epsilon}={\rm sup}\left\{(\left\lceil\frac{k_i}{n_i}\right\rceil-1)n_i: i\in R\right\}\geq {\rm sup}\left\{\left\lceil\frac{k_i}{n_i}\right\rceil-1: i\in R\right\}$, which implies ${\rm sup}\left\{(\left\lceil\frac{k_i}{n_i}\right\rceil-1)n_i: i\in R\right\}={\rm sup}\left\{\left\lceil\frac{k_i}{n_i}\right\rceil-1: i\in R\right\}$. It follows from (i) and (ii)  immediately that ${\rm I}(\mathcal{S})={\rm sup}\left\{(\left\lceil\frac{k_i}{n_i}\right\rceil-1)n_i: i\in R\right\}+{\rm D(G_{\mathcal{S}})}$, done.

Suppose that there exists some $\epsilon\in R_1$ such that \begin{equation}\label{equation k epsilon etc=sup}
(\left\lceil\frac{k_{\epsilon}}{n_{\epsilon}}\right\rceil-1)n_{\epsilon}={\rm sup}\left\{(\left\lceil\frac{k_i}{n_i}\right\rceil-1)n_i: i\in R\right\}
\end{equation}
with $\prod\limits_{i\in R_1\setminus \{\epsilon\}}n_i$ divides $\left\lceil\frac{k_{\epsilon}}{n_{\epsilon}}\right\rceil-1$.
To show ${\rm I}(\mathcal{S})={\rm sup}\left\{(\left\lceil\frac{k_i}{n_i}\right\rceil-1)n_i: i\in R\right\}+{\rm D(G_{\mathcal{S}})}$, by (ii) it suffices to construct an idempotent-sum free sequence $V\in \mathcal{F}(\mathcal{S})$ with length $|V|={\rm sup}\left\{(\left\lceil\frac{k_i}{n_i}\right\rceil-1)n_i: i\in R\right\}+{\rm D(G_{\mathcal{S}})}-1={\rm sup}\left\{(\left\lceil\frac{k_i}{n_i}\right\rceil-1)n_i: i\in R\right\}+N-1$. Let $\nu\in \mathcal{S}$ be such that ${\rm ind}_{g_i}(\nu(i))=1$ for all $i\in R$. Let $$V=\nu^{\left[{\rm sup}\left\{(\left\lceil\frac{k_i}{n_i}\right\rceil-1)n_i: i\in R\right\}+N-1\right]}.$$ To prove that $V$ is idempotent-sum free, we suppose to the contrary that $V$ contains a nonempty idempotent-sum subsequence, say $\nu^{[m]}$, where
\begin{equation}\label{equation m in}
1\leq m\leq {\rm sup}\left\{(\left\lceil\frac{k_i}{n_i}\right\rceil-1)n_i: i\in R\right\}+N-1.
\end{equation}
 It follows from \eqref{equation k epsilon etc=sup} and Lemma \ref{Lemma product condition containing idmepotent} that
 \begin{equation}\label{equation m geq}
 m=m\ {\rm ind}_{g_{\epsilon}}(\nu(\epsilon))\geq \left\lceil\frac{k_{\epsilon}}{n_{\epsilon}}\right\rceil n_{\epsilon}>{\rm sup}\left\{(\left\lceil\frac{k_i}{n_i}\right\rceil-1)n_i: i\in R\right\},
 \end{equation}
and that
$m=m\ {\rm ind}_{g_{i}}(\nu(i))\equiv 0\pmod {n_i}$ for all $i\in R_1$, equivalently,
\begin{equation}\label{equation m=0 mod N}
m\equiv 0\pmod N.
\end{equation} Recall that $\prod\limits_{i\in R_1\setminus \{\epsilon\}}n_i$ divides $\left\lceil\frac{k_{\epsilon}}{n_{\epsilon}}\right\rceil-1$.
 Then $N\mid (\left\lceil\frac{k_{\epsilon}}{n_{\epsilon}}\right\rceil -1) n_{\epsilon}={\rm sup}\left\{(\left\lceil\frac{k_i}{n_i}\right\rceil-1)n_i: i\in R\right\}$, combined with \eqref{equation m in}, \eqref{equation m geq} and \eqref{equation m=0 mod N} we derive a contradiction. Hence, $V$ is idempotent-sum free, done.

 Now, we consider the necessity. Suppose that ${\rm I}(\mathcal{S})={\rm sup}\left\{(\left\lceil\frac{k_i}{n_i}\right\rceil-1)n_i: i\in R\right\}+{\rm D(G_{\mathcal{S}})}$ and that there exists no $\epsilon\in R$ such that $(\left\lceil\frac{k_{\epsilon}}{n_{\epsilon}}\right\rceil-1)n_{\epsilon}={\rm sup}\left\{(\left\lceil\frac{k_i}{n_i}\right\rceil-1)n_i: i\in R\right\}$, with $\epsilon\in R\setminus R_1$ or $\prod\limits_{i\in R_1\setminus \{\epsilon\}}n_i$ divides $\left\lceil\frac{k_{\epsilon}}{n_{\epsilon}}\right\rceil-1$.
 Then $R_1\neq \emptyset$, which implies
 \begin{equation}\label{equation N>1}
 N>1.
 \end{equation}
  Let $T\in \mathcal{F}(\mathcal{S})$ be an idempotent-sum free sequence with $|T|={\rm I}(\mathcal{S})-1={\rm sup}\left\{(\left\lceil\frac{k_i}{n_i}\right\rceil-1)n_i: i\in R\right\}+{\rm D(G_{\mathcal{S}})}-1$.
Take a subsequence $U$ of $T$ such that $\Psi(U)$ is a zero-sum sequence, i.e., \begin{equation}\label{equation sum in all Ti zero}
\sum\limits_{a\mid U}{\rm ind}_{g_i}(a(i))\equiv 0\pmod {n_i}  \mbox{ for all }  i\in R,
\end{equation} with $|U|$ being {\bf maximal}.
Let $$W=T \cdot U^{[-1]}.$$
Since $T$ is idempotent-sum free, it follows from \eqref{equation sum in all Ti zero} and Lemma \ref{Lemma product condition containing idmepotent} that there exists some $\delta\in R$ such that $\sum\limits_{a\mid U}{\rm ind}_{g_{\delta}}(a(\delta))\leq(\left\lceil\frac{k_{\delta}}{n_{\delta}}\right\rceil-1)n_{\delta},$
and so $$|U|\leq \sum\limits_{a\mid U}{\rm ind}_{g_{\delta}}(a(\delta))\leq(\left\lceil\frac{k_{\delta}}{n_{\delta}}\right\rceil-1)n_{\delta}\leq {\rm sup}\left\{(\left\lceil\frac{k_i}{n_i}\right\rceil-1)n_i: i\in R\right\}.$$
On the other hand, by the maximality of $|U|$, we have that $\Psi(W)\in \mathcal{F}(G_{\mathcal{S}})$ is {\bf zero-sum free} and
$|W|=|\Psi(W)|\leq {\rm D(G_{\mathcal{S}})}-1$, equivalently,  $|U|=|T|-|W|\geq {\rm sup}\left\{(\left\lceil\frac{k_i}{n_i}\right\rceil-1)n_i: i\in R\right\}$. It follows that
\begin{equation}\label{equation sum|Ti|=}
|U|={\rm sup}\left\{(\left\lceil\frac{k_i}{n_i}\right\rceil-1)n_i: i\in R\right\}
\end{equation} and
\begin{equation}\label{equation |T cdot Ti-1|=}
|W|={\rm D(G_{\mathcal{S}})}-1=N-1.
\end{equation}

Now we show the following.

\noindent \textbf{Claim C.} \ All terms of $\Psi(T)$ are nonzero.

\noindent {\sl Proof of Claim C.} \ Assume to the contrary that there exists some term $z\mid T$ such that $\psi(z)$ is the zero element of $G_{\mathcal{S}}$. Since $\Psi(W)$ is zero-sum free, it follows that $z$ is one term of $U$. Since ${\rm ind}_{g_{i}}(z(i))$ is a positive multiple of $n_i$ for all $i\in R$ and there exists no $\epsilon\in R\setminus R_1$ such that $(\left\lceil\frac{k_{\epsilon}}{n_{\epsilon}}\right\rceil-1)n_{\epsilon}={\rm sup}\left\{(\left\lceil\frac{k_i}{n_i}\right\rceil-1)n_i: i\in R\right\}$, it follows from \eqref{equation sum|Ti|=} that for all $i\in R$,
$$\begin{array}{llll}
\sum\limits_{a\mid U}{\rm ind}_{g_i}(a(i))&=&
{\rm ind}_{g_i}(z(i))+\sum\limits_{a\mid U\cdot z^{[-1]}}{\rm ind}_{g_i}(a(i))  \\
&\geq &n_i+|U\cdot z^{[-1]}|\\
&=&n_i+{\rm sup}\left\{(\left\lceil\frac{k_r}{n_r}\right\rceil-1)n_r: r\in R\right\}-1\\
&>&(\left\lceil\frac{k_i}{n_i}\right\rceil-1)n_i,\\
\end{array}$$
It follows from \eqref{equation sum in all Ti zero} that $\sum\limits_{a\mid U}{\rm ind}_{g_i}(a(i))\geq \left\lceil\frac{k_i}{n_i}\right\rceil n_i$ for all $i\in R$, and follows from Lemma \ref{Lemma product condition containing idmepotent} that $U$ is a nonempty idempotent-sum subsequence of $T$, a contradiction with $T$ being idempotent-sum free. This proves Claim C. \qed

By \eqref{equation N>1}, \eqref{equation |T cdot Ti-1|=} and Claim C and by applying Lemma \ref{Lemma decomposition} with $\Psi(T)$, we conclude that
\begin{equation}\label{equation u and v same in T}
\psi(\alpha)=\psi(\beta) \mbox{ for any two terms } \alpha,\beta \mbox{ of } T.
\end{equation}

Since there exists no $\epsilon\in R$ such that $(\left\lceil\frac{k_{\epsilon}}{n_{\epsilon}}\right\rceil-1)n_{\epsilon}={\rm sup}\left\{(\left\lceil\frac{k_i}{n_i}\right\rceil-1)n_i: i\in R\right\}$ with $\epsilon\in R\setminus R_1$ or $\prod\limits_{i\in R_1\setminus \{\epsilon\}}n_i$ divides $\left\lceil\frac{k_{\epsilon}}{n_{\epsilon}}\right\rceil-1$,
we derive that $N$ does not divide ${\rm sup}\left\{(\left\lceil\frac{k_i}{n_i}\right\rceil-1)n_i: i\in R\right\}$. It follows that there exists some $$M\in [{\rm sup}\left\{(\left\lceil\frac{k_i}{n_i}\right\rceil-1)n_i: i\in R\right\}+1, \ \ {\rm sup}\left\{(\left\lceil\frac{k_i}{n_i}\right\rceil-1)n_i: i\in R\right\}+N-1]$$ with $N\mid M$. Take a subsequence $T'$ of $T$ of length $|T'|=M$. By \eqref{equation u and v same in T}, we have that $\Psi(T')$ is a zero-sum sequence, i.e.,  $\sum\limits_{a\mid T'}{\rm ind}_{g_i}(a(i))\equiv 0\pmod {n_i}$ for all $i\in R$. Moreover, we have that
$\sum\limits_{a\mid T'}{\rm ind}_{g_i}(a(i))\geq |T'|> (\left\lceil\frac{k_i}{n_i}\right\rceil-1)n_i$ and so $\sum\limits_{a\mid T'}{\rm ind}_{g_i}(a(i))\geq  \left\lceil\frac{k_i}{n_i}\right\rceil n_i$ for all $i\in R$. By Lemma \ref{Lemma product condition containing idmepotent}, we have that $T$ is not idempotent-sum free, a contradiction. This proves Conclusion (iii).
\end{proof}

\medskip

\begin{remark} We remark that all the bounds in Theorem \ref{theorem in product of cyclic semigroups} are sharp in general, for which the reasons are as follows.

By Conclusion (i) and (ii), we see that if the condition ${\rm D(G_{\mathcal{S}})}=1+\sum\limits_{i\in R_1}(n_i-1)$ holds true, then the lower bound ${\rm sup}\left\{(\left\lceil\frac{k_i}{n_i}\right\rceil-1)n_i: i\in R\right\}+1+ \sum\limits_{i\in  R_1}(n_i-1)$ and the upper bound ${\rm sup}\left\{(\left\lceil\frac{k_i}{n_i}\right\rceil-1)n_i: i\in R\right\}+{\rm D(G_{\mathcal{S}})}$ coincide, and thus, ${\rm I}(\mathcal{S})={\rm sup}\left\{(\left\lceil\frac{k_i}{n_i}\right\rceil-1)n_i: i\in R\right\}+{\rm D(G_{\mathcal{S}})}={\rm sup}\left\{(\left\lceil\frac{k_i}{n_i}\right\rceil-1)n_i: i\in R\right\}+1+ \sum\limits_{i\in  R_1}(n_i-1)$. In fact, some classes of finite abelian groups of form $G=\mathbb{Z}_{n_1}\oplus\cdots\oplus\mathbb{Z}_{n_r}$ with $1<n_1 \mid \cdots \mid n_r$ has been verified to satisfy ${\rm D(G)}=1+ \sum\limits_{i\in [1,r]}(n_i-1)$, including the case of finite abelian p-groups or the case of $r\leq 2$ (see Theorem D). Moreover, a significant open conjecture (see Conjecture 3.5 in \cite{GaoGeroldingersurvey}) in zero-sum theory is that the equality holds ture
\begin{equation}\label{equation conjecture r geq 3}
{\rm D(G)}=1+ \sum\limits_{i\in [1,r]}(n_i-1) \mbox{ when } r=3 \mbox{ or } n_1=\cdots =n_r,
\end{equation}
for which some theoretical evidence was provided (see \cite{GaoDis, GaoGeroldingersurveyEur} for example).

If there exists some $\delta\in R\setminus R_1$ such that the size of the component ${\rm C}_{k_{\delta};  n_{\delta}}$ of $\mathcal{S}$ as a product of cyclic semigroups is large enough,
precisely, $k_{\delta}-1=(\left\lceil\frac{k_{\delta}}{n_{\delta}}\right\rceil-1)n_{\delta}={\rm sup}\left\{(\left\lceil\frac{k_i}{n_i}\right\rceil-1)n_i: i\in R\right\}$, then as shown at the very beginning of the argument for Conclusion (iii), the lower bound ${\rm sup}\left\{\left\lceil\frac{k_i}{n_i}\right\rceil-1: i\in R\right\}+{\rm D(G_{\mathcal{S}})}$ coincides with the upper bound ${\rm sup}\left\{(\left\lceil\frac{k_i}{n_i}\right\rceil-1)n_i: i\in R\right\}+{\rm D(G_{\mathcal{S}})}$ and is definitely attained.
\end{remark}

 As noticed in the above remark that, if all $n_i\in R_1$ are equal and the conjecture shown in \eqref{equation conjecture r geq 3} holds, then the upper bound ${\rm sup}\left\{(\left\lceil\frac{k_i}{n_i}\right\rceil-1)n_i: i\in R\right\}+{\rm D(G_{\mathcal{S}})}$ is attained. To understand how closely that upper bound is getting, we considered the associated extremal problem in the opposite direction. That is, given that all $n_i\in R_1$ are pairwise coprime, could the upper bound ${\rm sup}\left\{(\left\lceil\frac{k_i}{n_i}\right\rceil-1)n_i: i\in R\right\}+{\rm D(G_{\mathcal{S}})}$ be attained and when? This is answered by Conclusion (iii) of Theorem \ref{theorem in product of cyclic semigroups}.

Well, in the rest of this paper, we shall prove the following fact that if the semigroup $\mathcal{S}$ is a direct product of infinitely many cyclic semigroups such that ${\rm I}(\mathcal{S})$ is finite, then we can find a {\sl finite} subsemigroup $\mathcal{S}'$ of $\mathcal{S}$ such that ${\rm I}(\mathcal{S})={\rm I}(\mathcal{S}')$, which can be seen in
Proposition \ref{prop periodic parts}. Furthermore, we shall deduce a conclusion to determine the precise value of ${\rm I}(\mathcal{S})$ when $\mathcal{S}$ is a product of finitely many cyclic semigroups with some constraints, which unifies Theorem D in the setting of commutative semigroups.

\medskip

\begin{prop}\label{prop periodic parts} \ For any nonempty set $R$, let $\mathcal{S}=\prod\limits_{i\in R}{\rm C}_{k_i;  n_i}$ where $k_i,n_i\geq 1$ for any $i\in R$.
Let $R_1=\{i\in R: n_i>1\}$. Suppose that $R\setminus R_1\neq \emptyset$, and that
${\rm I}(\mathcal{S})$ is finite, i.e., both $|R_1|$ and ${\rm sup}\ \{k_i:i\in R\setminus R_1\}$ are finite.  Then ${\rm I}(\mathcal{S})={\rm I}(\mathcal{S}')$
where  $\mathcal{S}'=\prod\limits_{i\in R_1\cup \{\epsilon\}}{\rm C}_{k_i;  n_i}$ with some $\epsilon\in R\setminus R_1$ satisfying $k_{\epsilon}={\rm sup}\ \{k_i:i\in R\setminus R_1\}$.
\end{prop}

\begin{proof} Let $\iota: \mathcal{S}'\rightarrow \mathcal{S}$ be a map defined by the following. For any $a'\in \mathcal{S}'$, let $\iota: a'\mapsto a$ with \begin{displaymath}
{\rm ind}_{g_i}(a(i))=\left\{ \begin{array}{ll}
{\rm ind}_{g_i}(a'(i)), & \textrm{if $i\in R_1\cup \{\epsilon\}$;}\\
k_i, & \textrm{otherwise.}\\
\end{array} \right.
\end{displaymath}
 Noticing that $C_{k_i;n_i}$ is a cyclic nilsemigroup when $i\in R\setminus R_1$, we can check that $\iota$ is a monomorphism and so $\mathcal{S}'$ is isomorphic to the subsemigroup $\iota(\mathcal{S}')$ of $\mathcal{S}$. It follows from Lemma \ref{lemma I(sub)< I(S)} that ${\rm I}(\mathcal{S}')={\rm I}(\iota(\mathcal{S}'))\leq {\rm I}(\mathcal{S})$.

Now it remains to show that ${\rm I}(\mathcal{S})\leq {\rm I}(\mathcal{S}')$. Take an idempotent-sum free sequence $T\in \mathcal{F}(\mathcal{S})$ of length ${\rm I}(\mathcal{S})-1$. Let $\widetilde{T}=\mathop{\bullet}\limits_{a\mid T} \tilde{a}\in \mathcal{F}(\mathcal{S})$, where
\begin{equation}\label{equation tilde{a}}
{\rm ind}_{g_i}(\tilde{a}(i))=\left\{ \begin{array}{ll}
{\rm ind}_{g_i}(a(i)), & \textrm{if $i\in R_1$;}\\
1, & \textrm{if $i=\epsilon$;}\\
k_i, & \textrm{otherwise.}\\
\end{array} \right.
\end{equation}
Observe that $\widetilde{T}\in \mathcal{F}(\iota(\mathcal{S}'))$.
Then we show the following.

\noindent \textbf{Claim D.} \ The sequence $\widetilde{T}$ is idempotent-sum free.

\noindent{\sl Proof of Claim D.} \ Suppose to the contrary that $\widetilde{T}$ is not idempotent-sum free, i.e., there exists a nonempty subsequence $L$ of $T$ such that $\widetilde{L}=\mathop{\bullet}\limits_{a\mid L} \tilde{a}$ is an idempotent-sum subsequence of $\widetilde{T}$. By Lemma \ref{Lemma product condition containing idmepotent} and \eqref{equation tilde{a}}, we have that
\begin{equation}\label{equation lenght of L}
|L|=\sum\limits_{a\mid L}{\rm ind}_{g_\epsilon}(\tilde{a}(\epsilon))\geq \left\lceil\frac{k_{\epsilon}}{n_{\epsilon}}\right\rceil n_{\epsilon}=k_{\epsilon},
\end{equation}
and that  for all  $i\in R_1$,
$$\sum\limits_{a\mid L}{\rm ind}_{g_i}(a(i))=\sum\limits_{a\mid L}{\rm ind}_{g_i}(\tilde{a}(i))
\geq  \left\lceil\frac{k_{i}}{n_{i}}\right\rceil n_{i} \ \mbox{ and } \  \sum\limits_{a\mid L}{\rm ind}_{g_i}(a(i))=\sum\limits_{a\mid L}{\rm ind}_{g_i}(\tilde{a}(i))\equiv 0\pmod {n_i}.$$
It follows from \eqref{equation lenght of L} that for all $i\in R\setminus R_1$,
$$\sum\limits_{a\mid L}{\rm ind}_{g_i}(a(i))\geq   |L|\geq k_{\epsilon}\geq k_i=\left\lceil\frac{k_{i}}{n_{i}}\right\rceil n_{i} \ \
\mbox{ and } \ \ \sum\limits_{a\mid L}{\rm ind}_{g_i}(a(i))\equiv 0\pmod {n_i=1}.$$ By Lemma \ref{Lemma product condition containing idmepotent}, we conclude that $L$ is a nonempty idempotent-sum subsequence of $T$, a contradiction with $T$ being idempotent-sum free. This proves Claim D.
\qed

By Claim D, we have that  ${\rm I}(\mathcal{S}')={\rm I}(\iota(\mathcal{S}'))\geq |\widetilde{T}|+1=|T|+1={\rm I}(\mathcal{S})$, completing the proof.
\end{proof}

By Theorem \ref{theorem in product of cyclic semigroups} and Theorem D, we have the following conclusion for the product of finitely many cyclic semigroups.

\begin{cor}\label{Corollary three conclusions} \ For $r\geq 1$, let $\mathcal{S}={\rm C}_{k_1;  n_1}\times \cdots \times {\rm C}_{k_r;  n_r}$.  Then $${\rm I}(\mathcal{S})=\max\limits_{i\in [1,r]}\left\{(\left\lceil\frac{k_i}{n_i}\right\rceil-1)n_i\right\}+{\rm D(G_{\mathcal{S}})}=\max\limits_{i\in [1,r]}\left\{(\left\lceil\frac{k_i}{n_i}\right\rceil-1)n_i\right\}+1+\sum\limits_{i=1}^r (n_i-1)$$ in case that one of the following conditions holds.

(i) $r=1$;

(ii) $r=2$ with $n_1\mid n_2$ or $n_2\mid n_1$;

(iii)  $\prod\limits_{i=1}^{r} n_i$ is a prime power.
\end{cor}

\begin{remark} In the above Corollary \ref{Corollary three conclusions}, when all $k_i=1$,  the semigroup $\mathcal{S}$ reduces to be the finite abelian group $\mathbb{Z}_{n_1}\oplus\cdots\oplus \mathbb{Z}_{n_r}$ and the conclusions reduces to be ones in Theorem D.
\end{remark}

\bigskip

\noindent {\bf Acknowledgements}

\noindent
This work is supported by NSFC (grant no. 11971347, 11501561).

\end{document}